\documentclass[11pt, a4paper,reqno]{amsart}
\usepackage{amsmath, amsthm, amsfonts, amssymb}
\usepackage{a4wide}
\usepackage[latin1]{inputenc}
\usepackage{graphicx}

\usepackage{enumitem}

\DeclareMathOperator{\reg}{reg}
\DeclareMathOperator{\SL}{SL}

\DeclareMathOperator{\N}{\mathbb{N}}
\DeclareMathOperator{\Z}{\mathbb{Z}}
\DeclareMathOperator{\R}{\mathbb{R}}
\DeclareMathOperator{\C}{\mathbb{C}}
\DeclareMathOperator{\Q}{\mathbb{Q}}
\renewcommand{\H}{\mathbb{H}}
\DeclareMathOperator{\tr}{tr}
\DeclareMathOperator{\imag}{Im}
\DeclareMathOperator{\calQ}{\mathcal{Q}}
\DeclareMathOperator{\erf}{erf}
\DeclareMathOperator{\erfc}{erfc}

\DeclareMathOperator{\sgn}{sgn}
\DeclareMathOperator{\even}{e}
\DeclareMathOperator{\odd}{o}

\renewcommand{\pmod}[1]{\  \,  \left( \mathrm{mod} \,  #1 \right)}

\makeatletter
\newcommand*\bigcdot{\mathpalette\bigcdot@{.5}}
\newcommand*\bigcdot@[2]{\mathbin{\vcenter{\hbox{\scalebox{#2}{$\m@th#1\bullet$}}}}}
\makeatother

\numberwithin{equation}{section}
	\newtheorem{Satz}{Satz}[section]
	\newtheorem{theorem}[Satz]{Theorem}
	\newtheorem{lemma}[Satz]{Lemma}
	\newtheorem{proposition}[Satz]{Proposition}

	\theoremstyle{definition} 
	
	\newtheorem{example}[Satz]{Example}
	\newtheorem{remark}[Satz]{Remark}
\date{\today}

\author{Claudia Alfes-Neumann}

\address{Mathematical Institute, Paderborn University, Warburger Str. 100,
D-33098 Paderborn, Germany}
\email{alfes@math.uni-paderborn.de}

\address{Mathematical Institute, University of Cologne, Weyertal 86-90, D--50931 Cologne, Germany}

\author{Kathrin Bringmann}
\email{kbringma@math.uni-koeln.de}

\author{Markus Schwagenscheidt}
\email{mschwage@math.uni-koeln.de}

\title{On the rationality of cycle integrals of meromorphic modular forms}

\thanks{The research of the second author is supported by the Alfried Krupp Prize for Young University Teachers of the Krupp foundation and the second and third author are supported by the SFB-TRR 191 ``Symplectic Structures in Geometry, Algebra and Dynamics'', funded by the DFG}

\begin{document} 

\begin{abstract}
	We derive finite rational formulas for the traces of cycle integrals of certain meromorphic modular forms. Moreover, we prove the modularity of a completion of the generating function of such traces. The theoretical framework for these results is an extension of the Shintani theta lift to meromorphic modular forms of positive even weight. 
\end{abstract}

\maketitle

\section{Introduction and statement of results}
\label{sec:introduction}

\subsection{Rationality of traces of cycle integrals of meromorphic cusp forms}
 Let $k$ be a positive even integer. While considering the Doi-Naganuma lift, Zagier \cite{zagierrealquadratic} encountered the functions
\begin{align*}
 f_{k,d}(z) :=\frac{|d|^{\frac{k+1}{2}}}{\pi} \sum_{\mathcal{Q} \in \mathcal{Q}_{d}} Q(z,1)^{-k},
\end{align*}
where $\mathcal{Q}_{d}$ denotes the set of all integral binary quadratic forms $Q = [a,b,c]$ of discriminant $d = b^{2}-4ac$.
These are holomorphic cusp forms of weight $2k$ for $\Gamma:=\SL_{2}(\Z)$ if $d > 0$, and meromorphic cusp forms of weight $2k$ for $\Gamma$ if $d < 0$, i.e., they are meromorphic modular forms which decay like cusp forms towards $i\infty$.

Kohnen and Zagier \cite{kohnenzagierrationalperiods} showed that certain simple linear combinations of the cycle integrals
\[
\int_{c_{Q}}f_{k,d}(z)Q(z,1)^{k-1}dz
\]
of the cusp forms $f_{k,d}$ for $d > 0$ are rational. Here $c_{Q} := \Gamma_{Q}\backslash C_{Q}$ is the image in $\Gamma \backslash \H$ of the geodesic 
\[
C_{Q} := \left \{z \in \H: a|z|^{2} + bx + c = 0 \right \}\qquad (z=x+iy) 
\]
associated to $Q = [a,b,c]\in \mathcal{Q}_{D}$ with $D > 0$. Complementing these results, we present rational formulas for the traces
\begin{align*}
\tr_{f_{k,\mathcal{A}}}(D) := \sum_{Q \in \mathcal{Q}_{D}/\Gamma}\int_{c_{Q}}f_{k,\mathcal{A}}(z)Q(z,1)^{k-1}dz
\end{align*}
of cycle integrals of the refined functions
\begin{align}\label{def fkA}
f_{k,\mathcal{A}} (z) := \frac{|d|^{\frac{k+1}{2}}}{\pi}\sum_{Q \in \mathcal{A}}Q(z,1)^{-k}, 
\end{align}
where $\mathcal{A} \in \mathcal{Q}_{d}/\Gamma$ is a fixed equivalence class of quadratic forms of discriminant $d < 0$. The poles of $f_{k,\mathcal{A}}$ lie at the CM points $z_{Q}\in\H$ for $Q \in \mathcal{A}$, which are characterized by $Q(z_{Q},1) = 0$. We assume that they do not lie on any of the geodesics $C_{Q}$ for $Q \in \mathcal{Q}_{D}$. Let $z_{\mathcal{A}} := x_{\mathcal{A}} + iy_{\mathcal{A}} \in \H$ denote a fixed CM point $z_{Q}$ for some $Q \in \mathcal{A}$. We obtain the following rationality result for the traces of $f_{k,\mathcal{A}}$.

\begin{theorem}\label{corollary trace formulas}
	Let $F$ be a weakly holomorphic modular form of weight $\frac{3}{2}-k$ for $\Gamma_{0}(4)$ satisfying the Kohnen plus space condition. Suppose that the Fourier coefficients $a_{F}(-D)$ vanish for all $D > 0$ which are squares and that $a_F(-D)$ is rational for $D>0$. Moreover, assume that $z_{\mathcal{A}}$ does not lie on any of the geodesics $C_{Q}$ for $Q \in \mathcal{Q}_{D}$ for any $D>0$ for which $a_{F}(-D) \neq 0$. Then the linear combinations
	\[
	\sum_{D > 0}a_{F}(-D)\tr_{f_{k,\mathcal{A}}}(D) 
	\]
	are rational.
\end{theorem}

We compute some numerical values of the above traces in Example~\ref{example trace formulas} below. Theorem~\ref{corollary trace formulas} follows from the following explicit formulas for the traces.

 \begin{theorem}\label{theorem trace formulas}
	Assume the hypotheses of Theorem~\ref{corollary trace formulas}. Then we have the formula
	\begin{align*}
	&\sum_{D > 0}a_{F}(-D)\tr_{f_{k,\mathcal{A}}}(D)\!=\! \frac{\sqrt{|d|}}{\,\left|\overline{\Gamma}_{z_{\mathcal{A}}}\right|}\sum_{D > 0}a_{F}(-D)\!\left(\frac{c_{k}(D)}{y_{\mathcal{A}}^{k-1}}+4\left(i\sqrt{D}\right)^{k-1}\!\!\!\!\!\!\sum_{\substack{Q =[a,b,c]\in \mathcal{Q}_{D} \\ Q_{z_{\mathcal{A}} > 0 > a}}}\!\!\!P_{k-1}\left(\frac{iQ_{z_{\mathcal{A}}}}{\sqrt{D}}\right) \right),
	\end{align*}
	where $P_{\ell}$ is the $\ell$-th Legendre polynomial, $|\overline{\Gamma}_{z_{\mathcal{A}}}|$ is the order of the stabilizer of $z_{\mathcal{A}}$ in $\overline{\Gamma} := \Gamma/\{\pm 1\}$, and
\begin{align*}
Q_{z} := \frac{1}{y}\left(a|z|^{2}+bx+c\right)
\end{align*}
	for $Q =[a,b,c]$. The constant $c_{k}(D)$ is given by
	\begin{align}\label{eq ckD}
	c_{k}(D) := \frac{D^{k-\frac{1}{2}}\zeta(k)}{2^{k-3}(2k-1)\zeta(2k)}L_{D_{0}}(k)\sum_{m \mid f}\mu(m)\left(\frac{D_{0}}{m} \right)m^{-k}\sigma_{1-2k}\left(\frac{f}{m}\right).
	\end{align}
Here we write  $D = D_{0}f^{2}$ with a fundamental discriminant $D_{0}$, $\mu$ is the Möbius function, $\sigma_{\kappa}(n) := \sum_{d|n} d^\kappa$ is the $\kappa$-th divisor sum, $\zeta$ is the Riemann zeta function, $( \frac{D_{0}}{\cdot})$ is the Kronecker symbol, and $L_{D_{0}}(s)$ is the associated Dirichlet $L$-function.  In particular, $c_k(D)$ is rational.
\end{theorem}

\begin{remark}
	For $D > 0$ a non-square discriminant and $Q = [a,b,c] \in \mathcal{Q}_{D}$ the geodesic $C_{Q}$ is a semi-circle centered at the real line. The condition $Q_{z_{\mathcal{A}}} > 0 > a$ means that $C_{Q}$ is oriented clockwise and that $z_{\mathcal{A}}$ lies in the interior of the bounded component of $\H\setminus C_{Q}$. Since, for fixed non-square $D > 0$, every point $z \in \H$ lies in the interior of the bounded component of $\H \setminus C_{Q}$ for only finitely many $Q \in \mathcal{Q}_{D}$, the sum over $Q \in \mathcal{Q}_{D}$ in Theorem~\ref{theorem trace formulas} is finite. 
\end{remark}

The proof of Theorem~\ref{theorem trace formulas} relies on the Fourier expansion of a certain theta lift of the meromorphic modular forms $f_{k,\mathcal{A}}$, which is explained in Section~\ref{section shintani lift fourier expansion introduction}. An outline of the proof of Theorem~\ref{theorem trace formulas} can be found in Section~\ref{section outline proof}, and the full proof  is given  in Section~\ref{section proof trace formulas}. 

To illustrate Theorem~\ref{corollary trace formulas} and Theorem~\ref{theorem trace formulas}, we treat two examples in low weights.

\begin{example}\label{example trace formulas}
	We consider the cases $k\in\{2,4\}$, since then the space $S_{2k}$ of cusp forms of weight $2k$ for $\Gamma$ is trivial. By the Shimura correspondence, the space of cusp forms of weight $k+\frac{1}{2}$ for $\Gamma_{0}(4)$ in the Kohnen plus space  is isomorphic to $S_{2k}$, and hence trivial as well. This implies that for every discriminant $D > 0$ there exists a weakly holomorphic modular form $F$ of weight $\frac{3}{2}-k$ for $\Gamma_{0}(4)$ satisfying the Kohnen plus space condition such that $a_{F}(-D) = 1$ and $a_{F}(\ell) = 0$ for $-D \neq \ell < 0$. Suppose that $D > 0$ is a non-square discriminant and that $z_{\mathcal{A}}$ does not lie on any of the geodesics $C_{Q}$ for $Q \in \mathcal{Q}_{D}$. Using that $P_1(x)=x$ and $P_3(x)=\frac52 x^3-\frac32 x$ we obtain the special cases
	\begin{align*}
	\tr_{f_{2,\mathcal{A}}}(D) &= \frac{\sqrt{|d|}}{\left|\overline{\Gamma}_{z_{\mathcal{A}}}\right|}\left(\frac{c_{2}(D)}{y_{\mathcal{A}}} - 4 \!\!\!\sum_{\substack{Q = [a,b,c] \in \mathcal{Q}_{D} \\ Q_{z_{\mathcal{A}}} > 0 > a}}Q_{z_{\mathcal{A}}}\right),\\
	\tr_{f_{4,\mathcal{A}}}(D) &= \frac{\sqrt{|d|}}{\left|\overline{\Gamma}_{z_{\mathcal{A}}}\right|}\left(\frac{c_{4}(D)}{y_{\mathcal{A}}^{3}} - 2 \!\!\! \sum_{\substack{Q = [a,b,c] \in \mathcal{Q}_{D} \\ Q_{z_{\mathcal{A}}} > 0 > a}}\left(5Q_{z_{\mathcal{A}}}^{3}+3DQ_{z_{\mathcal{A}}}\right) \right).
	\end{align*}
	 In the following table, we give some numerical values for the $D$-th trace of $f_{2,[1,1,1]}$ and $f_{4,[1,1,1]}$. To clear the denominators, we display the values of $\tr_{f_{2,[1,1,1]}}(D)$ and $3\tr_{f_{4,[1,1,1]}}(D)$.
	\begin{align*}
	\begin{array}{c|c|c|c|c|c|c|c|c|c|c|c|c|c}
	D & 5 & 8 & 13 & 17 & 20 & 21 & 24 & 29 & 32 & 33 & 37 & 40 & 41  \\
	\hline\hline
	\tr_{f_{2,[1,1,1]}}(D) & 4 & 8 & 12 & 28 & 24 & 20 & 32 & 20 & 40 & 64 & 44 & 64 & 76  \\
	\hline
	3\tr_{f_{4,[1,1,1]}}(D) & 20 & 48 & 92 & 452 & 320 & 340 & 576 & 260 & 880 & 1664 & 1596 & 1920 & 2612 
	\end{array}
	\end{align*}
	We left out those discriminants $D$ which are squares, and $D = 12$ and $D = 28$ since in these cases the CM point $z_{A}$ lies on one of the geodesics of discriminant $D$. We explain the numerical evaluation in Section~\ref{section numerical evaluation}. We checked the above values using Sage, by computing the cycle integrals using numerical integration. As the values in the table above suggest, for $k \in \{2,4\}$ the numbers $|\overline{\Gamma}_{z_{\mathcal{A}}}|\cdot|d|^{\frac{k}{2}-1}\cdot \tr_{f_{k,\mathcal{A}}}(D)$ are always even integers (for any $\mathcal{A}$), which is not hard to show using the formula \eqref{eq rational formula ckD} for $c_{k}(D)$.
\end{example}

\subsection{The regularized Shintani theta lift of a meromorphic cusp form} \label{section shintani lift introduction}
	We now describe the theoretical foundation of our work. The classical Shimura-Shintani correspondence establishes a Hecke equivariant isomorphism between the spaces of cusp forms of half-integral weight $k + \frac{1}{2}$ and even integral weight $2k$, with $k \in \N := \{1,2,3,\dots\}$. Soon after its discovery by Shimura \cite{shimura}, this correspondence was realized by  Niwa \cite{niwa} and Shintani \cite{shintani} as a theta lift, that is, as an integral constructed from a theta kernel in two variables. The classical Shintani theta lift for cusp forms was recently generalized to weakly holomorphic modular forms by Guerzhoy, Kane, and the second author \cite{bringmannguerzhoykane}, to harmonic Maass forms by the first and the third author \cite{ans}, and to differentials of the third kind by Bruinier, Funke, Imamoglu, and Li \cite{bifl}. Extending the results of \cite{bifl}, we also include meromorphic cusp forms of arbitrary positive even weight with poles of arbitrary order in the upper half-plane in the Shintani theta lift. 

For $k \in \N$ we let $\mathbb{S}_{2k}$ denote the space of meromorphic cusp forms of weight $2k$ for $\Gamma$. Every meromorphic modular form can be written as a sum of a meromorphic cusp form, a weakly holomorphic modular form, and, if $k = 1$, a multiple of $j'/j$, with $j$ the usual modular $j$-invariant.  Since the Shintani theta lifts of weakly holomorphic modular forms and of the meromorphic modular form $j'/j$ have already been determined in \cite{ans, bifl}, we restrict our attention to theta lifts of meromorphic cusp forms. 
 
For $k \in \N$ and a fundamental discriminant $\Delta \in \Z$ satisfying $(-1)^{k}\Delta > 0$ we let $\Theta_{k,\Delta}(z,\tau)$ denote the Shintani theta function defined in \eqref{Thet}. The function $\Theta_{k,\Delta}(-\overline{z},\tau)$ is real-analytic in both variables and transforms like a modular form of weight $2k$ in $z$ for $\Gamma$ and of weight $k+\frac{1}{2}$ in $\tau$ for $\Gamma_{0}(4)$. We define the regularized Shintani theta lift of $f \in \mathbb{S}_{2k}$ by
\begin{align}\label{eq Shintani lift introduction}
\Phi_{k,\Delta}(f,\tau) := \left\langle f, \overline{ \Theta_{k,\Delta}(\,\cdot\,,\tau)} 
\right\rangle^{\reg},
\end{align}
where the regularized inner product is defined in \eqref{eq inner product}. The regularized integral in \eqref{eq Shintani lift introduction} exists by the following theorem.

\begin{theorem}\label{theorem regularization}
	For $f \in \mathbb{S}_{2k}$ the Shintani theta lift $\Phi_{k,\Delta}(f,\tau)$ is a real-analytic function on $\H$ that transforms like a modular form of weight $k + \frac{1}{2}$ for $\Gamma_{0}(4)$ and satisfies the Kohnen plus space condition.
\end{theorem}

More generally, Proposition \ref{proposition regularized inner product} shows that the regularized inner product in \eqref{eq Shintani lift introduction} converges even if we replace $\overline{\Theta_{k,\Delta}(\,\cdot\,, \tau)}$ by any real-analytic function $g$ which transforms like a modular form of weight $2k$ and is of moderate growth at $i\infty$.

\subsection{The Fourier expansion of the Shintani theta lift}\label{section shintani lift fourier expansion introduction} One of the main results of this paper is the Fourier expansion of the Shintani theta lift $\Phi_{k,\Delta}(f,\tau)$ of $f \in \mathbb{S}_{2k}$.  It turns out that $\Phi_{k,\Delta}(f,\tau)$ yields a completion of the generating function of twisted traces of cycle integrals
\begin{align*}
\tr_{f,\Delta}(D) := \sum_{Q \in \mathcal{Q}_{|\Delta|D}/\Gamma}\chi_{\Delta}(Q)\int_{c_{Q}}^{\reg}f(z)Q(z,1)^{k-1}dz,
\end{align*}
where $\chi_{\Delta}$ is the usual genus character as defined on page 238 of \cite{kohnenfouriercoefficients}, and the cycle integrals have to be regularized as explained in Section~\ref{section cycle integrals} if poles of $f$ lie on the geodesic $c_{Q}$.

To describe the non-holomorphic part of the Shintani theta lift, we define for $z \in \H, v > 0$, and any integral binary quadratic form $Q= [a,b,c]$ the function
\begin{align}\label{eq xi preimage}
\phi_{Q}(z,v) := \frac{\sqrt{|\Delta|}}{4}Q(z,1)^{k-1}\left(\sgn(Q_{z})-\erf\left(2\sqrt{\pi v} \frac{Q_{z}}{\sqrt{|\Delta|}} \right)\right),
\end{align}
where $\erf(x) \!:= \!\frac{2}{\sqrt{\pi}}\int_{0}^{x}e^{-t^{2}}dt$ is the error function and $\sgn(0) := 0$.
Note that $z\mapsto\phi_{Q}(z,v)$ is real-analytic up to a jump singularity along $C_{Q}$ if $Q$ has positive discriminant. More generally, for any $n \in \N_{0}$ and $z \in \H, v > 0$, we consider the function 
\begin{align*}
&R_{2-2k,z}^{n}\left(\phi_{Q}(z,v)\right) \\
&\quad:= \frac{\sqrt{|\Delta|}}{4}\left(\sgn(Q_{z})R_{2-2k,z}^{n}\left(Q(z,1)^{k-1}\right)- R_{2-2k,z}^{n}\left(Q(z,1)^{k-1}\erf\left(2\sqrt{\pi v} \frac{Q_{z}}{\sqrt{|\Delta|}} \right)\right)\right),
\end{align*}
 where $R_{\kappa}^{n} := R_{\kappa+2n-2}\circ\dots\circ R_{\kappa}$ is an iterated version of the Maass raising operator $R_{\kappa} := 2i\frac{\partial}{\partial z}+\frac{\kappa}{y}$. Furthermore, we define for $n \in \N_{0}$ and $z,\tau = u + iv \in \H$ the 
``theta function''
\[
R_{2-2k,z}^{n}\left(\theta^{*}_{1-k,\Delta}(z,\tau) \right) := \sum_{D \in \Z}\sum_{Q \in \mathcal{Q}_{|\Delta|D}}\chi_{\Delta}(Q)R_{2-2k,z}^{n}\left(\phi_{Q}(z,v)\right)e^{2\pi i D\tau}.
\]

We are now ready to state the Fourier expansion of the Shintani theta lift.
\begin{theorem}\label{theorem fourier expansion}
	Let $f \in \mathbb{S}_{2k}$. Then the Fourier expansion of the Shintani theta lift of $f$ is given by
	\begin{align*} 
	\Phi_{k,\Delta}(f,\tau) &= \frac{\sqrt{|\Delta|}}{2}\sum_{D > 0}\tr_{f,\Delta}(D)e^{2\pi i D \tau} \\
	&\quad +  (-4)^{1-k}\pi\sum_{\substack{\varrho \in \Gamma \backslash \H}}\frac{1}{\left|\overline{\Gamma}_{\varrho}\right|}\sum_{n \geq 1}c_{f,\varrho}(-n)\frac{\imag(\varrho)^{n-2k}}{(n-1)!}\left[R_{2-2k,z}^{n-1}\left(\theta^{*}_{1-k,\Delta}(z,\tau) \right)\right]_{z = \varrho},
	\end{align*}
	where $c_{f,\varrho}(\ell)$ denotes the $\ell$-th coefficient in the elliptic expansion \eqref{eq elliptic expansion} of $f$ at $\varrho \in \mathbb{H}$ and $\overline{\Gamma}_{\varrho}$ is the stabilizer of $\varrho$ in $\overline{\Gamma} = \Gamma/\{\pm 1\}$.
\end{theorem}

\begin{remark}\ 
	\begin{enumerate}[leftmargin=*]
		\item The sum over $\varrho \in \Gamma \backslash \H$ only runs over the finitely many poles of $f$ modulo $\Gamma$.
		\item Using a vector-valued setting as in \cite{ans,bifl}, the methods of the present paper can be applied to compute the Shintani theta lift of meromorphic cusp forms for $\Gamma_{0}(N)$.
		\item Let $\Delta = 1$ and $k \in \N$ even. The authors of \cite{bringmannkanezwegers} constructed a function $\widehat{\Psi}(z,\tau)$ on $\H \times \H$ that is real-analytic in both variables and transforms like a modular form of weight $2-2k$ in $z$ for $\Gamma$ and of weight $k+\frac{1}{2}$ in $\tau$ for $\Gamma_{0}(4)$. Comparing their construction with Theorem~\ref{theorem fourier expansion}, it is not hard to see that the Shintani theta lift $\Phi_{k,1}(f,\tau)$ of $f \in \mathbb{S}_{2k}$ is, up to addition of a holomorphic cusp form, given by a linear combination of the functions  $[R_{2-2k,z}^{n-1}(\widehat{\Psi}(z,\tau))]_{z = \varrho}$, with $\varrho$ running over the poles of $f$ in $\Gamma \backslash\H$ and $n \in \N$ with $c_{f,\varrho}(-n) \neq 0$. 
	\end{enumerate}
\end{remark}

%Next, we show that the Shintani theta lift of a meromorphic cusp form is related to $\Theta_{k,\Delta}^{*}(z,\tau)$, defined in \eqref{MTheta}. A straightforward calculation yields the following result.
%
%
%\begin{corollary}\label{corollary xi image}
%	Under the lowering operator $L := -2iv^{2}\frac{\partial}{\partial \overline{\tau}} \ (\tau=u+iv)$ the Shintani theta lift $\Phi_{k,\Delta}(f,\tau)$ of $f \in \mathbb{S}_{2k}$ maps to
%	\begin{align*}
%	L\left(\Phi_{k,\Delta}(f,\tau)\right) = (-1)^k 2^{1-2k}\pi \sum_{\substack{\varrho \in \Gamma \backslash \H}}\frac{1}{\left|\overline{\Gamma}_{\varrho}\right|}\sum_{n \geq 1}c_{f,\varrho}(-n)\frac{\imag(\varrho)^{n-2k}}{(n-1)!}\left[R_{2-2k,z}^{n-1}\left(\Theta^{*}_{k,\Delta}(z,\tau)\right)\right]_{z = \varrho}.
%	\end{align*}
%\end{corollary}

\subsection{Outline of the proof of Theorem~\ref{theorem trace formulas}}\label{section outline proof}

We now briefly explain how the above results about the Shintani theta lift of meromorphic modular forms imply the formula in Theorem~\ref{theorem trace formulas}. For the details of the proof we refer to Section~\ref{section proof trace formulas}. We let $k$ be even and $\Delta = 1$, and drop $\Delta$ from the notation. For simplicity, we assume $S_{2k} = \{0\}$, that is, $k \in \{2,4\}$. 

The proof of Theorem~\ref{theorem trace formulas} uses the theory of locally harmonic Maass forms introduced by Kane, Kohnen and the second author in \cite{bringmannkanekohnen}, together with an intimate interplay between the Shintani theta lift of meromorphic cusp forms studied in this work and another theta lift investigated by Kane, Viazovska, and the second author in \cite{brikavia}. For a harmonic Maass form $\tau \mapsto g(\tau)$ of weight $\frac{3}{2}-k$ for $\Gamma_{0}(4)$ satisfying the Kohnen plus space condition this theta lift is defined by a regularized integral
\[
\Phi^{*}_{k}(g,z) := \int_{\Gamma_{0}(4)\backslash \H}^{\reg}g(\tau)\Theta_{k}^{*}(z,\tau)\frac{dudv}{v^{2}},
\]
where $\Theta_{k}^{*}(z,\tau)$ is a theta function which has weight $k-\frac{3}{2}$ in $\tau$ and weight $2-2k$ in $z$ (see \eqref{MTheta} for its definition), and the regularized integral is defined as in \eqref{eq integral PD 1} below. In \cite{brikavia}, the authors showed that for $D > 0$ the theta lift of the $D$-th weakly holomorphic Poincar\'e series $P_{\frac{3}{2}-k,D}(\tau) = \frac{2}{3}q^{-D}+O(1)$ of weight $\frac{3}{2}-k$ is a multiple of a so-called locally harmonic Maass form $\mathcal{F}_{1-k,D}(z)$, that is,
\begin{align}\label{eq1}
\Phi^{*}_{k}\left(P_{\frac{3}{2}-k,D},z\right) \stackrel{\bigcdot}{=}\mathcal{F}_{1-k,D}(z),
\end{align}
where the symbol $\stackrel{\bigcdot}{=}$ means equality up to some non-zero constant factor. The function $\mathcal{F}_{1-k,D}$, which was first studied in \cite{bringmannkanekohnen}, transforms like a modular form of weight $2-2k$ for $\Gamma$ and is harmonic on $\H$ up to jump singularities along the geodesics $C_{Q}$ for $Q \in \mathcal{Q}_{D}$. Moreover, it turns out that in the case $S_{2k} = \{0\}$ the locally harmonic Maass form $\mathcal{F}_{1-k,D}$ is in fact locally a polynomial, which has been explicitly determined in \cite{bringmannkanekohnen}.

On the other hand, let us consider the Shintani theta lift $\Phi_{k}(f_{k,\mathcal{A}},\tau)$ of the meromorphic cusp form $f_{k,\mathcal{A}}$ for a class $\mathcal{A} \in \mathcal{Q}_{d}/\Gamma$ with $d < 0$. A short calculation using the Fourier expansion from Theorem~\ref{theorem fourier expansion} shows that the lowering operator $L := -2iv^{2}\frac{\partial}{\partial \overline{\tau}}$ acts on $\Phi_{k}(f_{k,\mathcal{A}},\tau)$ by
\begin{align}\label{eq2}
L\left( \Phi_{k}(f_{k,\mathcal{A}},\tau)\right) \stackrel{\bigcdot}{=} \left[R_{2-2k,z}^{k-1}\left( \Theta_{k}^{*}(z,\tau)\right)\right]_{z = z_{\mathcal{A}}},
\end{align}
where $z_{\mathcal{A}}$ is a fixed CM point $z_{Q}$ for some $Q \in \mathcal{A}$. We apply the iterated raising operator to \eqref{eq1} and plug in \eqref{eq2} to obtain
\begin{align*}
\int_{\Gamma_{0}(4)\backslash \H}^{\reg}P_{\frac{3}{2}-k,D}(\tau)L\left(\Phi_{k}(f_{k,\mathcal{A}},\tau)\right)\frac{dudv}{v^{2}} \stackrel{\bigcdot}{=} \left[R_{2-2k}^{k-1}\left(\mathcal{F}_{1-k,D}(z) \right)\right]_{z = z_{\mathcal{A}}}.
\end{align*}
The integral on the left-hand side can now be evaluated by a standard argument using Stokes' Theorem. It turns out that it essentially equals the $D$-th Fourier coefficient of $\Phi_{k}(f_{k,\mathcal{A}},\tau)$, which is the $D$-th trace of cycle integrals of $f_{k,\mathcal{A}}$. Hence we arrive at
\begin{align*}
\tr_{f_{k,\mathcal{A}}}(D) \stackrel{\bigcdot}{=} \left[R_{2-2k}^{k-1}\left(\mathcal{F}_{1-k,D}(z) \right)\right]_{z = z_{\mathcal{A}}}.
\end{align*}
The action of the iterated raising operator on the locally polynomial function $\mathcal{F}_{1-k,D}$ can easily be calculated, which yields the formula in Theorem~\ref{theorem trace formulas}.

\subsection{Organization of the paper} We start with a section on the necessary preliminaries about theta functions, regularized inner products, and cycle integrals of meromorphic cusp forms. In the remaining part of this work, we give the proofs of the above results. To prove Theorem~\ref{theorem regularization} we derive elliptic expansions of real-analytic functions on $\H$ and study regularized inner products of meromorphic and real-analytic modular forms in Section~\ref{section regularized integrals}. The computation of the Fourier expansion from Theorem~\ref{theorem fourier expansion} is explained in Section~\ref{section fourier expansion}. In Section~\ref{section proof trace formulas}, we give the proof of Theorem~\ref{corollary trace formulas} and Theorem~\ref{theorem trace formulas}. Finally, in Section~\ref{section numerical evaluation} we provide some details on the numerical evaluation of the rational formulas for the traces from Theorem~\ref{theorem trace formulas}.

\section*{Acknowledgements}
We thank Jan Bruinier for helpful discussions and Stephan Ehlen, Jens Funke, Pavel Guerzhoy, Chris Jennings-Shaffer, Ken Ono, and Shaul Zemel for comments on an earlier version of this paper. Moreover, we thank the referee for helpful comments.

\section{Preliminaries}
\label{section preliminaries}

\subsection{Theta functions}
\label{section theta functions}

 For $k \in \N$ and $\Delta \in \Z$ a fundamental discriminant satisfying $(-1)^{k}\Delta > 0$, the Shintani theta function is defined as
\begin{align}\label{Thet}
\Theta_{k,\Delta}(z,\tau) := y^{-2k}v^{\frac{1}{2}}\sum_{D \in \Z}\sum_{Q \in \calQ_{|\Delta|D}}\chi_{\Delta}(Q)Q(z,1)^{k}e^{-4\pi v \frac{Q_{z}^{2}}{|\Delta|}}e^{2\pi i D \tau}.
\end{align}
Note that $\Theta_{k,\Delta}(z,\tau)$ would vanish identically if $(-1)^{k}\Delta < 0$. The function $\Theta_{k,\Delta}(-\overline{z},\tau)$ is real-analytic in both variables and transforms like a modular form of weight $2k$ in $z$ for $\Gamma$ and weight $k + \frac{1}{2}$ in $\tau$ for $\Gamma_{0}(4)$ (see Proposition~3.2 of \cite{brikavia}). Moreover, as a function of $z$ it has moderate growth at $i\infty$ (see Proposition~4.2 of \cite{ans}). 

	We also consider the Millson theta function
	\begin{align}\label{MTheta}
	\Theta^{*}_{k,\Delta}(z,\tau) := v^{\frac32}\sum_{D \in \Z}\sum_{Q \in \mathcal{Q}_{|\Delta|D}}\chi_{\Delta}(Q)Q_{z}Q(z,1)^{k-1} e^{-4\pi v \frac{Q_{z}^{2}}{|\Delta|}}e^{2\pi i D \tau},
	\end{align}
	which transforms like a modular form of weight $2-2k$ in $z$ for $\Gamma$ and weight $k-\frac{3}{2}$ in $\tau$ for $\Gamma_{0}(4)$ (see Proposition~3.2 of \cite{brikavia}).
	
\subsection{Regularized inner products}
\label{section inner products}
Next, we describe the regularized inner product in \eqref{eq Shintani lift introduction}, which was first introduced by Petersson \cite{petersson} and later rediscovered and extended by Harvey and Moore \cite{harveymoore}, Borcherds \cite{borcherds}, Bruinier \cite{bruinierhabil}, and others. We denote by $[\varrho_{1}],\dots,[\varrho_{r}] \in \Gamma \backslash \H$ the equivalence classes of all of the poles of $f$ on $\H$ and we choose a fundamental domain $\mathcal{F}^{*}$ for $\Gamma \backslash \H$ such that $\varrho_{\ell} \in \overline{\Gamma}_{\varrho_{\ell}}\mathcal{F}^{*}$ for all $1 \leq \ell \leq r$. For any $\varrho \in \H$ and $\varepsilon > 0$ we consider the $\varepsilon$-ball around $\varrho$, 
\begin{align*}
B_{\varepsilon}(\varrho) := \left \{z \in \H: |X_{\varrho}(z)| < \varepsilon \right \}, \qquad X_{\varrho}(z):= \frac{z-\varrho}{z-\overline{\varrho}}.
\end{align*}
Let $g: \mathbb{H} \to \mathbb{C}$ be real-analytic and assume that $g$ transforms like a modular form of weight $2k$ for $\Gamma$ and is of moderate (i.e., polynomial) growth at $i\infty$. We define the regularized Petersson inner product of $f$ and $g$ by
\begin{align}\label{eq inner product}
\left\langle f, g \right\rangle^{\reg}  := \lim_{\varepsilon_{1},\dots,\varepsilon_{r} \to 0}\int_{\mathcal{F}^{*}\setminus \bigcup_{\ell=1}^{r}B_{\varepsilon_{\ell}}(\varrho_{\ell})}f(z)\overline{g(z)}y^{2k}\frac{dxdy}{y^{2}}.
\end{align}
We see in Proposition~\ref{proposition regularized inner product} that the regularized inner product exists.

		\begin{remark}
		Similar regularized inner products in the case that both $f$ and $g$ are meromorphic cusp forms or weakly holomorphic modular forms have recently been studied, for example, in \cite{bringmannkanevonpippich, zemel}. 
	\end{remark}

\subsection{Regularized cycle integrals of meromorphic cusp forms}
\label{section cycle integrals}

Let $D > 0$ and let $Q = [a,b,c]\in \mathcal{Q}_{D}$. The associated geodesic $C_{Q}$ is a semi-circle centered at the real line if $a \neq 0$, and a vertical line if $a = 0$. We orient it counterclockwise if $a > 0$ and from $-\frac{c}{b}$ to $i\infty$ if $a=0$ and $b>0$. If poles of $f$ lie on $C_{Q}$, then we modify it as follows. For every pole $\varrho$ of $f$ lying on $C_{Q}$ choose $\varepsilon > 0$ sufficiently small such that no other poles of $f$ lie on $B_{\varepsilon}(\varrho)$. We denote by $C_{Q,\varepsilon}^{\pm}$ the path that agrees with $C_{Q}$ outside of every such ball but circumvents every pole $\varrho$ of $f$ along the boundary arc of $B_{\varepsilon}(\varrho)$ that lies in the connected component of $\H \setminus C_{Q}$ with $\pm Q_{z} > 0$. Moreover $c_{Q}:=\Gamma_{Q}\backslash C_{Q}$ is the image of $C_{Q}$ in the modular curve $\Gamma \backslash \H$ and  $c_{Q,\varepsilon}^{\pm} := \Gamma_{Q} \backslash C_{Q,\varepsilon}^{\pm}$.

We define the regularized geodesic cycle integral of $f \in \mathbb{S}_{2k}$ along $c_{Q}$ by
\begin{equation*}
\int_{c_{Q}}^{\reg}f(z)Q(z,1)^{k-1}dz := \frac{1}{2}\lim_{\varepsilon \to 0}\left(\int_{c_{Q,\varepsilon}^{+}}f(z)Q(z,1)^{k-1}dz+\int_{c_{Q,\varepsilon}^{-}}f(z)Q(z,1)^{k-1}dz\right).
\end{equation*}
This is sometimes also called the Cauchy principal value of the geodesic cycle integral, see Section 2.4 of \cite{bifl}. Furthermore, if no pole of $f$ lies on $c_{Q}$, then the above definition agrees with the usual definition of geodesic cycle integrals.

\section{Proof of Theorem~\ref{theorem regularization}}
\label{section regularized integrals}

A meromorphic function $f: \H \to \C$ has an elliptic expansion near $\varrho \in \H$ of the shape
\begin{align}\label{eq elliptic expansion}
f(z) = (z-\overline{\varrho})^{-2k}\sum_{n \gg -\infty}c_{f,\varrho}(n)X_{\varrho}^{n}(z), 
\end{align}
with coefficients $c_{f,\varrho}(n) \in \C$. For a proof see Proposition~17 in Zagier's part of  \cite{zagier123}, note that the required modularity of $f$ in the cited proposition is actually not necessary.

 In order to prove Theorem~\ref{theorem regularization}, we would like to plug in elliptic expansions of $f$ and $\Theta_{k,\Delta}(z,\tau)$ near the poles of $f$. Note that $ z \mapsto \Theta_{k,\Delta}(z,\tau)$ is real-analytic. The shape of elliptic expansions of real-analytic functions are described in the following lemma.

\begin{lemma}\label{lemma elliptic expansion}
		Let $\kappa\in\Z$, $\varrho \in \H$, and $g: \H \to \C$ be real-analytic near $\varrho$. Then $g$ has an elliptic expansion of the shape (near $\varrho$)
		\begin{align}\label{eq elliptic expansion real analytic}
		g(z) = (z-\overline{\varrho})^{-\kappa}\sum_{n\in \Z}c_{g,\varrho}(|X_{\varrho}(z)|,n)X_{\varrho}^{n}(z),
		\end{align}
		with coefficients $c_{g,\varrho}(r,n) \in \C$, which are analytic as functions of the real variable $r$. Near $r=0$ we have the Taylor expansion
		\begin{align}\label{eq elliptic coefficients real analytic}
		c_{g,\varrho}(r,n) = \sum_{m\geq \max\{0,-n\}}a_{g,\varrho,n}(m)r^{2m},
		\end{align}
		with coefficients $a_{g,\varrho,n}(m) \in \C$ given in \eqref{eq elliptic expansion 2}. The constant term $c_{g,\varrho}(0,n)$ is given by
		\begin{align}\label{eq elliptic constant coefficient}
		c_{g,\varrho}(0,n) = \begin{cases} 
		\frac{1}{n!}(2i)^{\kappa}\imag(\varrho)^{n+\kappa}R_{\kappa}^{n}(g)(\varrho) & \text{if } n \geq 0 ,\\
		0 & \text{if } n < 0.
		\end{cases}
		\end{align}
	\end{lemma}

	\begin{proof}
		We generalize the proof of Proposition~17 in Zagier's part in \cite{zagier123}. Expanding as a Taylor series we obtain
		\[
		g(\varrho+W) =\sum_{a,b \geq 0} \left[\frac{\partial^{b}}{\partial \overline {z}^{b}}\frac{\partial^{a}}{\partial z^{a}}g(z)\right] _{z=\varrho}\frac{W^{a}}{a!}\frac{\overline{W}^{b}}{b!}.
		\]
		Writing $\frac{\varrho - \overline{\varrho}w}{1-w}= \varrho + \frac{2i\varrho_2 w}{1-w}$ ($\varrho_2:=\text{Im}(\varrho)$) we obtain, for $|w|$ sufficiently small, the formula
		\begin{align}\label{eq elliptic expansion 1}
		(1-w)^{-\kappa}g\left(\frac{\varrho-\overline{\varrho}w}{1-w}\right) = (1-w)^{-\kappa}\sum_{a,b \geq 0}\left[\frac{\partial^{b}}{\partial \overline {z}^{b}}\frac{\partial^{a}}{\partial z^{a}}g(z)\right]_{z=\varrho}\frac{\left(\frac{2i\varrho_2 w}{1-w}\right)^{a}}{a!}\frac{\left(-\frac{2i\varrho_2\overline{w}}{1-\overline{w}}\right)^{b}}{b!}.
		\end{align}
		Expanding $(1-w)^{-\kappa-a}$ and $(1-\overline{w})^{-b}$ using the Binomial Theorem, the right-hand side of \eqref{eq elliptic expansion 1} becomes 
		\begin{align*}
		\sum_{a,b, j, \ell \geq 0}\binom{\kappa+a+j-1}{j}\binom{b+\ell-1}{\ell}\left[\frac{\partial^{b}}{\partial \overline {z}^{b}}\frac{\partial^{a}}{\partial z^{a}}g(z)\right]_{z=\varrho}\frac{(2i\varrho_2)^{a}}{a!}\frac{(-2i\varrho_2)^{b}}{b!}w^{a+j-b-\ell}|w|^{2b+2\ell}.
		\end{align*}
		We reorder the summation by setting $n = a + j-b-\ell$ and $m = b+\ell$. Plugging in $w = X_{\varrho}(z)$ and using the formulas 
		\[
		1-w = \frac{2i\varrho_2}{z-\overline{\varrho}}, \qquad \frac{\varrho-\overline{\varrho}w}{1-w} = z,
		\]
		we obtain the elliptic expansion
		\begin{align*}
		g(z) &= (z-\overline{\varrho})^{-\kappa}\sum_{n \in \Z}\left(\sum_{m \geq \max\{0,-n\}}a_{g,\varrho,n}(m)|X_{\varrho}(z)|^{2m}\right)X_{\varrho}^{n}(z)
		\end{align*}
		with coefficients
		\begin{multline}\label{eq elliptic expansion 2}
		\hspace{-1em}a_{g,\varrho,n}(m) \\
		:= \left(2i\varrho_{2}\right)^{\kappa}\sum_{\substack{0 \leq a \leq m+n \\ 0 \leq b \leq m}}\binom{\kappa+m+n-1}{m+n-a}\binom{m-1}{m-b}\left[\frac{\partial^{a}}{\partial z^{a}}\frac{\partial^{b}}{\partial \overline{z}^{b}}g(z)\right]_{z=\varrho}\frac{(2i\varrho_2)^{a}}{a!}\frac{(-2i\varrho_2)^{b}}{b!}.
		\end{multline}
		The formula for $c_{g,\varrho}(0,n) = a_{g,\varrho,n}(0)$ follows from \eqref{eq elliptic expansion 2} and (56) in Zagier's part in \cite{zagier123}.
	\end{proof}
	
	Using the elliptic expansion of a real-analytic function given in Lemma~\ref{lemma elliptic expansion}, we can now prove that the regularized Petersson inner product of a meromorphic cusp form and a real-analytic modular form exists.
	
	\begin{proposition}\label{proposition regularized inner product}
		The regularized Petersson inner product, defined in \eqref{eq inner product}, exists.
	\end{proposition}

	\begin{proof}
		We divide $\mathcal{F}^{*}$ into a compact domain containing all of the poles $\varrho_{1},\dots ,\varrho_{r}$ of $f$, and a remaining set on which $f$ is holomorphic. The integral over the second set converges since $f$ decays like a cusp form towards $i\infty$ and $g$ is of moderate growth at $i\infty$. Hence it suffices to show that for every pole $\varrho:=\varrho_{\ell}$ of $f$ and $\delta > 0$ sufficiently small the limit
		\[
		\lim_{\varepsilon \to 0}\int_{A_{\varepsilon}^{\delta}(\varrho)\cap \mathcal{F}^{*}}f(z)\overline{g(z)}y^{2k}\frac{dxdy}{y^{2}}, \qquad A_{\varepsilon}^{\delta}(\varrho) := B_{\delta}(\varrho)\setminus \overline{B_{\varepsilon}(\varrho)},
		\]
		 exists. In order to obtain an integral over the whole annulus $A_{\varepsilon}^{\delta}(\varrho)$, we recall that, by (2a.15) of \cite{petersson50}, we have the disjoint union
\begin{align}\label{eq disjoint union}
B_{\varepsilon}(\varrho) = \dot{\bigcup}_{M \in \overline{\Gamma}_{\varrho}}M\left(B_{\varepsilon}(\varrho)\cap \mathcal{F}^{*}\right).
\end{align}
Therefore we can write, using the modularity of $f$ and $g$,
		\begin{align}\label{splitunion}
		\int_{A_{\varepsilon}^{\delta}(\varrho)\cap \mathcal{F}^{*}}f(z)\overline{g(z)}y^{2k}\frac{dxdy}{y^{2}} = \frac{1}{\left|\overline{\Gamma}_{\varrho}\right|}\int_{A_{\varepsilon}^{\delta}(\varrho)}f(z)\overline{g(z)}y^{2k}\frac{dxdy}{y^{2}}.
	\end{align}
		We plug in the elliptic expansions of $f$ from \eqref{eq elliptic expansion} and $g$ from \eqref{eq elliptic expansion real analytic} with $\kappa = 2k$ to rewrite the right-hand side of \eqref{splitunion} as 
		\begin{align*}
		\frac{1}{\left|\overline{\Gamma}_{\varrho}\right|}\int_{A_{\varepsilon}^{\delta}(\varrho)}\sum_{\substack{n \gg -\infty \\ m\in\Z}}c_{f,\varrho}(n)\overline{c_{g,\varrho}(|X_{\varrho}(z)|,m)}X_{\varrho}^{n}(z)\overline{X_{\varrho}^{m}(z)}\frac{y^{2k}}{|z-\overline{\varrho}|^{4k}}\frac{dxdy}{y^{2}}.
		\end{align*}
		We next make the change of variables $X_{\varrho}(z) = e^{i\vartheta}r$ with $0 < \vartheta < 2\pi$ and $\varepsilon < r <\delta$. Then a short calculation shows that
		\[
		\frac{y^{2k}}{|z-\overline{\varrho}|^{4k}} = \left(\frac{1-r^{2}}{4\varrho_2}\right)^{2k}, \qquad \frac{dxdy}{y^{2}} = \frac{4r}{(1-r^{2})^{2}}d\vartheta dr,
		\]
		where again $\varrho_2=\imag(\varrho)$. Thus
		\begin{equation}\label{rewriteint}
		\frac{4}{(4\varrho_2)^{2k}\left|\overline{\Gamma}_{\varrho}\right|}\int_{\varepsilon}^{\delta}\int_{0}^{2\pi}\sum_{\substack{n \gg -\infty \\ m\in\Z}}c_{f,\varrho}(n)\overline{c_{g,\varrho}(r,m)}e^{i\vartheta(n-m)}r^{m+n+1}\left(1-r^{2}\right)^{2k-2}d\vartheta dr.
		\end{equation}
		The integral over $\vartheta$ vanishes unless $m = n$, in which case it equals $2\pi$, thus \eqref{rewriteint} becomes
		\[
		\frac{8\pi}{(4\varrho_2)^{k}\left|\overline{\Gamma}_{\varrho}\right|}\int_{\varepsilon}^{\delta}\sum_{n\gg -\infty}c_{f,\varrho}(n)\overline{c_{g,\varrho}(r,n)}r^{2n+1}\left(1-r^{2}\right)^{2k-2} dr.
		\]
		The limit as $\varepsilon \to 0$ of the contribution from $n \geq 0$ clearly exists. Hence we need to show that the integral
		\[
		\int_{0}^{\delta}\overline{c_{g,\varrho}(r,n)}r^{2n+1}\left(1-r^{2}\right)^{2k-2} dr
		\]
		exists for the finitely many $n < 0$ with $c_{f,\varrho}(n) \neq 0$. Indeed, the integral exists since for $n < 0$ we have $c_{g,\varrho}(r,n) = O(r^{|2n|})$ by \eqref{eq elliptic coefficients real analytic}. This finishes the proof.
	\end{proof}
	We are now ready to prove Theorem \ref{theorem regularization}.

	\begin{proof}[Proof of Theorem~\ref{theorem regularization}]
	Noting that $g(z) = \overline{\Theta_{k,\Delta}(z,\tau)}$ is real-analytic in $z$ and of moderate growth at $i\infty$, Proposition~\ref{proposition regularized inner product} shows that the regularized Shintani theta lift \eqref{eq Shintani lift introduction} exists and hence transforms like a modular form of weight $k + \frac{1}{2}$ for $\Gamma_{0}(4)$. The fact that the Shintani theta lift is real-analytic follows from its Fourier expansion given in Theorem~\ref{theorem fourier expansion}. This finishes the proof of Theorem~\ref{theorem regularization}.
\end{proof}

\section{Proof of Theorem~\ref{theorem fourier expansion}}
\label{section fourier expansion}

In this section we compute the Fourier expansion of the Shintani theta lift and thereby prove Theorem~\ref{theorem fourier expansion}. 

\subsection{Preliminary computations}
We start with a useful formula which can be viewed as a generalization of the Residue Theorem.

\begin{lemma}\label{lemma residue theorem}
		Let $\varrho \in \H$, let $f:\mathbb H\to \C$ be meromorphic near $\varrho$, and let $g:\mathbb H\to\C$ be real-analytic near $\varrho$. Then we have the formula
		\[
		\lim_{\varepsilon \to 0}\int_{\partial B_{\varepsilon}(\varrho)}f(z)g(z)dz = \frac{\pi}{\imag(\varrho)}\sum_{n < 0}c_{f,\varrho}(n)c_{g,\varrho}(0,-n-1),
		\]
		where $c_{f,\varrho}(n)$ and $c_{g,\varrho}(r,n)$ are the coefficients of the elliptic expansions \eqref{eq elliptic expansion} of $f$ and \eqref{eq elliptic expansion real analytic} of $g$ at $\varrho$ (with $\kappa = 2-2k$).
	\end{lemma}
	
	\begin{proof}
		For $\varepsilon > 0$ sufficiently small we can plug in the elliptic expansions of $f$ and of $g$ to obtain that
		\begin{align*}
		\int_{\partial B_{\varepsilon}(\varrho)}f(z)g(z)dz &= \int_{\partial B_{\varepsilon}(\varrho)}(z-\overline{\varrho})^{-2}\sum_{\substack{n \gg -\infty\\ m \in \Z}}c_{f,\varrho}(n)c_{g,\varrho}(|X_{\varrho}(z)|,m)X_{\varrho}^{m+n}(z) dz \\
		&= \sum_{\substack{n \gg -\infty\\ m \in \Z}}c_{f,\varrho}(n)c_{g,\varrho}(\varepsilon,m)\int_{\partial B_{\varepsilon}(\varrho)}\frac{(z-\varrho)^{m+n}}{(z-\overline{\varrho})^{m+n+2}} dz.
		\end{align*}
		In the last step we use that $|X_{\varrho}(z)| = \varepsilon$ on $\partial B_{\varepsilon}(\varrho)$. By the Residue Theorem the last integral vanishes unless $m = -n-1$, in which case it equals $\frac{2\pi i}{\varrho-\overline{\varrho}}$. Taking the limit $\varepsilon \to 0$ and using that $c_{g,\varrho}(0,-n-1) = 0$ for $n \geq 0$ by \eqref{eq elliptic constant coefficient}, we obtain the stated formula.
	\end{proof}
	
	To ease notation, define 
	\[
	\varphi_{Q}(z,v):= y^{-2k}v^{\frac{1}{2}}Q(z,1)^{k}e^{-4\pi v \frac{Q_{z}^{2}}{|\Delta|}},
	\]
	so that the Shintani theta function can be written as
	\[
	\Theta_{k,\Delta}(z,\tau) = \sum_{D \in \Z}\sum_{Q \in \mathcal{Q}_{|\Delta|D}}\chi_{\Delta}(Q)\varphi_{Q}(z,v)e^{2\pi i D\tau}.
	\]
	A short calculation, using that $\frac{\partial}{\partial \overline{z}}Q_z=-y^2\frac{i}{2} Q(z,1)$, yields the following result.
	
	\begin{lemma} \label{lemma xi preimage}
	For $z \in \H$ with $Q_{z} \neq 0$ and all $v > 0$ the function $\phi_Q$, defined in \eqref{eq xi preimage}, satisfies 
	\[
	L_{z}\left( \phi_{Q}(z,v)\right) = \varphi_{Q}(z,v)y^{2k}.
	\]
	\end{lemma}
	
	\begin{remark} The function $\phi_{Q}(z,v)$ was used in \cite{ans} to compute the Shintani theta lift of harmonic Maass forms and in \cite{bifl} for $k = 1$ to compute the Shintani theta lift of meromorphic modular forms of weight two that are holomorphic at the cusps and have at most simple poles in $\H$. 
	\end{remark} 
	Using Lemma~\ref{lemma xi preimage}, we see that the $D$-th Fourier coefficient of the regularized Shintani theta lift \eqref{eq Shintani lift introduction} of $f \in \mathbb{S}_{2k}$ with poles at $[\varrho_{1}],\dots,[\varrho_{r}] \in \Gamma \backslash \H$ is given by
	\begin{align}\label{eq Dth coefficient}
	\lim_{\varepsilon_{1},\dots,\varepsilon_{r} \to 0}\int_{\mathcal{F}^{*}\setminus \bigcup_{\ell=1}^{r}B_{\varepsilon_{\ell}}(\varrho_{\ell})}f(z)L_z\left(\sum_{Q \in \mathcal{Q}_{|\Delta|D}}\chi_{\Delta}(Q)\phi_{Q}(z,v)\right)\frac{dx dy}{y^{2}}.
	\end{align}
	We compute the Fourier coefficients for $D \leq 0$ and $D > 0$ separately.
	\subsection{The coefficients of index $D \leq 0$.} For $D \leq 0$ and $Q \in \mathcal{Q}_{|\Delta|D}$ the function $z\mapsto \phi_{Q}(z,v)$ is real-analytic. Using Stokes' Theorem, we find that \eqref{eq Dth coefficient} equals
\begin{align}\label{eq stokes theorem 1}
-\lim_{\varepsilon_{1},\dots,\varepsilon_{r} \to 0}\int_{\partial \left(\mathcal{F}^{*}\setminus \bigcup_{\ell = 1}^{r}B_{\varepsilon_{\ell}}(\varrho_{\ell}) \right)}f(z)\sum_{Q \in \mathcal{Q}_{|\Delta|D}}\chi_{\Delta}(Q)\phi_{Q}(z,v)dz.
\end{align}
Here we also use the fact that $f$ decays like a cusp form towards $i\infty$. If $\varepsilon_{1},\dots,\varepsilon_{r}$ are sufficiently small, then the boundary of $\mathcal{F}^{*}\setminus \bigcup_{\ell = 1}^{r}B_{\varepsilon_{\ell}}(\varrho_{\ell})$ consists of a disjoint union of the boundary arcs $-\partial B_{\varepsilon_{\ell}}(\varrho_{\ell}) \cap \mathcal{F}^{*}$ for $1 \leq \ell \leq r$ and further remaining boundary pieces, which come in $\Gamma$-equivalent pairs with opposite orientation and hence cancel out in the integral due to the modularity of the integrand. Therefore, \eqref{eq stokes theorem 1} equals 
\begin{align}\label{eq stokes theorem 2}
\sum_{1\leq \ell \leq r}\sum_{Q \in \mathcal{Q}_{|\Delta|D}}\chi_{\Delta}(Q)\lim_{\varepsilon_{\ell} \to 0}\int_{\partial B_{\varepsilon_{\ell}}(\varrho_{\ell}) \cap \mathcal{F}^{*}}f(z)\phi_{Q}(z,v)dz.
\end{align}
Using the disjoint union \eqref{eq disjoint union} again, we can rewrite \eqref{eq stokes theorem 2} as
\begin{align*}
\sum_{1\leq \ell \leq r}\frac{1}{\left|\overline{\Gamma}_{\varrho_{\ell}}\right|}\sum_{Q \in \mathcal{Q}_{|\Delta|D}}\chi_{\Delta}(Q)\lim_{\varepsilon_{\ell} \to 0}\int_{\partial B_{\varepsilon_{\ell}}(\varrho_{\ell})}f(z)\phi_{Q}(z,v)dz.
\end{align*}
	Since $\phi_{Q}(z,v)$ is real-analytic at $z = \varrho_\ell$, Lemma~\ref{lemma residue theorem} combined with \eqref{eq elliptic constant coefficient} yields that
	\begin{multline}\label{eq integral evaluation}
	\lim_{\varepsilon_{\ell} \to 0}\int_{\partial B_{\varepsilon_{\ell}}(\varrho_{\ell})}f(z)\phi_{Q}(z,v)dz 
	\\= (-4)^{1-k}\pi\sum_{n \geq 1}c_{f,\varrho_{\ell}}(-n)\frac{\imag(\varrho_{\ell})^{n-2k}}{(n-1)!}
	\left[R_{2-2k,z}^{n-1}\left(\phi_{Q}(z,v) \right)\right]_{z = \varrho_{\ell}}.
	\end{multline}
	We obtain the coefficients of index $D \leq 0$ as stated in Theorem~\ref{theorem fourier expansion}.

	\subsection{The coefficients of index $D > 0$} 
	We first assume that the poles of $f$ do not lie on any of the geodesics $C_{Q}$ for $Q \in \mathcal{Q}_{|\Delta|D}$. Since $\varphi_{Q}(z,v)$ has a jump singularity along the geodesic $C_{Q}$, one cannot directly apply Stokes' Theorem. Hence, for $\delta > 0$ and every $Q \in \mathcal{Q}_{|\Delta|D}$ we first cut out a $\delta$-tube
\[
U_{\delta}(C_{Q}) :=  \left \{z \in \H: |Q_{z}| < \delta \right \}
\]
around $C_{Q}$ from $\mathcal{F}^{*}$. Applying Stokes' Theorem, we can rewrite \eqref{eq Dth coefficient} as
\begin{align}\label{eq coefficient D positive}
\begin{split}
&\sum_{Q \in \mathcal{Q}_{|\Delta|D}}\chi_{\Delta}(Q)\lim_{\delta \to 0}\int_{\partial U_{\delta}(C_{Q}) \cap \mathcal{F}^{*}}f(z)\phi_{Q}(z,v)dz  \\
&\quad +\sum_{ 1 \leq \ell \leq r}\frac{1}{\left|\overline{\Gamma}_{\varrho_{\ell}}\right|}\sum_{Q \in \mathcal{Q}_{|\Delta|D}}\chi_{\Delta}(Q)\lim_{\varepsilon_{\ell} \to 0}\int_{\partial B_{\varepsilon_{\ell}}(\varrho_{\ell})}f(z)\phi_{Q}(z,v)dz.
\end{split}
\end{align} 
Since $z\mapsto \phi_{Q}(z,v)$ is real-analytic at every pole of $f$ by assumption, the integral in the second term can be evaluated using Lemma~\ref{lemma residue theorem} as in \eqref{eq integral evaluation} above. The first term in \eqref{eq coefficient D positive} is computed in the following proposition.

\begin{proposition}
	Assume that no pole of $f$ lies on any geodesic $C_{Q}$ for $Q \in \mathcal{Q}_{|\Delta|D}$. Then we have 
	\begin{align}\label{inttrace}
	\sum_{Q \in \mathcal{Q}_{|\Delta|D}}\chi_{\Delta}(Q)\lim_{\delta \to 0}\int_{\partial U_{\delta}(C_{Q}) \cap \mathcal{F}^{*}}f(z)\phi_{Q}(z,v)dz = \frac{\sqrt{|\Delta|}}{2}\tr_{f,\Delta}(D).
	\end{align}
\end{proposition}

\begin{proof}
By definition \eqref{eq xi preimage} the left-hand side of \eqref{inttrace} equals
\begin{multline}\label{eq inttrace1}
\frac{\sqrt{|\Delta|}}{4}\sum_{Q \in \mathcal{Q}_{|\Delta|D}}\chi_{\Delta}(Q)\\
\times \lim_{\delta \to 0}\int_{\partial U_{\delta}(C_{Q}) \cap \mathcal{F}^{*}}f(z)Q(z,1)^{k-1}
\left(\sgn(Q_{z})-\erf\left(2\sqrt{\pi v}\frac{Q_{z}}{\sqrt{|\Delta|}}\right)\right)dz.
\end{multline}
The path $\partial U_{\delta}(C_{Q}) \cap \mathcal{F}^{*}$ can be divided into two paths which are distinguished by the sign of $Q_{z}$. The orientation of $C_{Q}$ is given such that the connected component of $\H \setminus C_{Q}$ with $Q_{z} > 0$ lies on the right in the direction of travel. The part of $\partial U_{\delta}(C_{Q}) \cap \mathcal{F}^{*}$ which lies in the component of $\H \setminus C_{Q}$ with $Q_{z} > 0$ has the same orientation. Using $\erf(0) = 0$, the expression in \eqref{eq inttrace1} is therefore given by
\begin{align}\label{computelimit}
\frac{\sqrt{|\Delta|}}{2}\sum_{Q \in \mathcal{Q}_{|\Delta|D}}\chi_{\Delta}(Q)\int_{C_{Q} \cap \mathcal{F}^{*}}f(z)Q(z,1)^{k-1}dz.
\end{align}
We split $Q \in \mathcal{Q}_{|\Delta|D}$ as $Q\circ M$ with $Q \in \mathcal{Q}_{|\Delta|D}/\Gamma$ and $M \in \Gamma_{Q}\backslash \Gamma$, and rewrite \eqref{computelimit} as
\begin{align*}
&\frac{\sqrt{|\Delta|}}{2}\sum_{Q \in \mathcal{Q}_{|\Delta|D}/\Gamma}\sum_{M \in \Gamma_{Q} \backslash \Gamma}\chi_{\Delta}(Q\circ M)\int_{C_{Q\circ M} \cap \mathcal{F}^{*}}f(z)(Q\circ M)(z,1)^{k-1}dz \\
&\quad=\frac{\sqrt{|\Delta|}}{2}\sum_{Q \in \mathcal{Q}_{|\Delta|D}/\Gamma}\chi_{\Delta}(Q)\sum_{M \in \Gamma_{Q} \backslash \Gamma}\int_{C_{Q} \cap M\mathcal{F}^{*}}f(z)Q(z,1)^{k-1}dz \\
&\quad= \frac{\sqrt{|\Delta|}}{2}\sum_{Q \in \mathcal{Q}_{|\Delta|D}/\Gamma}\chi_{\Delta}(Q)\int_{\Gamma_{Q}\backslash C_{Q}}f(z)Q(z,1)^{k-1}dz.
\end{align*}
For the first equality we use that $\chi_{\Delta}(Q\circ M) = \chi_{\Delta}(Q)$, $(Q\circ M)(z,1) = Q(z,1)|_{-2}M$, and $C_{Q \circ M} = M^{-1} C_{Q}$. We obtain the formula as stated in the proposition.
\end{proof}

We next consider the case that the poles of $f$ do lie on a geodesic $C_{Q}$ for some $Q \in \mathcal{Q}_{|\Delta|D}$.  By similar arguments as above we find that \eqref{eq Dth coefficient} for $D > 0$ equals
\begin{align}\label{eq coefficient D positive with poles}
\begin{split}
&\frac{\sqrt{|\Delta|}}{2}\sum_{\substack{Q \in \mathcal{Q}_{|\Delta|D} \\ \varrho_{1},\dots,\varrho_{r} \not\in C_{Q}}}\chi_{\Delta}(Q)\int_{C_{Q} \cap \mathcal{F}^{*}}f(z)Q(z,1)^{k-1}dz  \\
&\quad +\sum_{\substack{Q \in \mathcal{Q}_{|\Delta|D} \\ \varrho_{1},\dots,\varrho_{r} \notin C_{Q}} }\chi_{\Delta}(Q)\sum_{ 1\leq \ell \leq r}\frac{1}{\left|\overline{\Gamma}_{\varrho_{\ell}}\right|}\lim_{\varepsilon_{\ell} \to 0}\int_{\partial B_{\varepsilon_{\ell}}(\varrho_{\ell})}f(z)\phi_{Q}(z,v)dz \\
&\quad +\sum_{\substack{Q \in \mathcal{Q}_{|\Delta|D} \\ \varrho_{\ell} \in C_{Q} \\\text{for some $\ell$}} }\chi_{\Delta}(Q)\lim_{\varepsilon_{1},\dots,\varepsilon_{r} \to 0} \lim_{\delta\to 0}\int_{\partial \left(U_{\delta}(C_{Q})\cup \bigcup_{\ell = 1}^{r}B_{\varepsilon_{\ell}}(\varrho_{\ell})\right)\cap \mathcal{F}^{*}}f(z)\phi_{Q}(z,v)dz.
\end{split}
\end{align}
As in \eqref{eq integral evaluation}, the second term of \eqref{eq coefficient D positive with poles} can again be computed using Lemma~\ref{lemma residue theorem}. By plugging in the definition \eqref{eq xi preimage} of $\phi_{Q}(z,v)$, we can rewrite the third term of \eqref{eq coefficient D positive with poles} as
\begin{align}\label{eq coefficient D positive with poles 2}
&\frac{\sqrt{|\Delta|}}{4}\sum_{\substack{Q \in \mathcal{Q}_{|\Delta|D} \\ \varrho_{\ell} \in C_{Q} \\\text{for some $\ell$}} }\chi_{\Delta}(Q)\lim_{\varepsilon \to 0}\left(\int_{C_{Q,\varepsilon}^{+}\cap \mathcal{F}^{*}}f(z)Q(z,1)^{k-1}dz + \int_{C_{Q,\varepsilon}^{-}\cap \mathcal{F}^{*}}f(z)Q(z,1)^{k-1}dz\right)  \nonumber \\
&\quad- \frac{\sqrt{|\Delta|}}{4}\sum_{\substack{Q \in \mathcal{Q}_{|\Delta|D} \\ \varrho_{\ell} \in C_{Q} \\\text{for some $\ell$}} }\chi_{\Delta}(Q)\sum_{\substack{1\leq \ell \leq r \\ \varrho_{\ell} \in C_{Q}}}\frac{1}{\left|\overline{\Gamma}_{\varrho_{\ell}}\right|}\lim_{\varepsilon_{\ell} \to 0}\int_{\partial B_{\varepsilon_{\ell}}(\varrho_{\ell})}f(z)Q(z,1)^{k-1}\erf\left( 2\sqrt{\pi v}\frac{Q_{z}}{\sqrt{|\Delta|}}\right)dz \nonumber\\
&\quad + \frac{\sqrt{|\Delta|}}{4}\sum_{\substack{Q \in \mathcal{Q}_{|\Delta|D} \\ \varrho_{\ell} \in C_{Q} \\\text{for some $\ell$}} }\chi_{\Delta}(Q)\sum_{\substack{1\leq \ell \leq r \\ \varrho_{\ell} \notin C_{Q}}}\frac{1}{\left|\overline{\Gamma}_{\varrho_{\ell}}\right|}\lim_{\varepsilon_{\ell} \to 0}\int_{\partial B_{\varepsilon_{\ell}}(\varrho_{\ell})}f(z)\phi_{Q}(z,v)dz,
\end{align}
where $C^{\pm}_{Q,\varepsilon}$ is the modified geodesic defined in Section~\ref{section cycle integrals}. The first term of \eqref{eq coefficient D positive with poles 2} combines with the first term of \eqref{eq coefficient D positive with poles} to the $D$-th $\Delta$-twisted trace of the regularized cycle integrals of $f$. Since the integrands in the second and third line of \eqref{eq coefficient D positive with poles 2} are real-analytic near $\varrho_{\ell}$, the integrals can again be evaluated using Lemma~\ref{lemma residue theorem}. We obtain that the second and third term of \eqref{eq coefficient D positive with poles 2} combine to
\begin{align*}
 &(-4)^{1-k}\pi\sum_{\substack{Q \in \mathcal{Q}_{|\Delta|D} \\ \varrho_{\ell} \in C_{Q} \\\text{for some $\ell$}} }\chi_{\Delta}(Q)\sum_{ 1\leq \ell \leq r}\frac{1}{\left|\overline{\Gamma}_{\varrho_{1}}\right|}\sum_{n < 0}c_{f,\varrho_{\ell}}(n)\frac{\imag(\varrho_{\ell})^{-n-2k}}{(-n-1)!}\left[R_{2-2k,z}^{-n-1}\left(\phi_{Q}(z,v)\right)\right]_{z = \varrho_{\ell}}.
\end{align*}
This finishes the proof of Theorem~\ref{theorem fourier expansion}.
 
 \section{Proofs of Theorem~\ref{corollary trace formulas} and Theorem~\ref{theorem trace formulas}}
 \label{section proof trace formulas}
Throughout this section we let $\Delta = 1$ and drop it from the notation. Furthermore, let $d < 0$ be a discriminant and let $\mathcal{A} \in \mathcal{Q}_{d}/\Gamma$ be a fixed class of quadratic forms.

We consider the Shintani theta lift of the meromorphic cusp form $f_{k,\mathcal{A}}$. A short calculation verifies the formula
\[
Q(z,1) = \frac{\sqrt{|d|}}{2\imag(z_{Q})}(z-\overline{z_{Q}})^{2}X_{z_{Q}}(z)
\] 
	for every $Q \in \mathcal{Q}_{d}$. In particular, the only poles of $f_{k,\mathcal{A}}$ in $\H$ lie at the CM points $z_{Q}$ for $Q \in \mathcal{A}$, and the elliptic expansion of $f_{k,\mathcal{A}}$ at $z_{Q}$ has the shape 
	\[
	f_{k,\mathcal{A}}(z) = (z-\overline{z_{Q}})^{-2k}\left(\frac{\sqrt{|d|}}{\pi}\left(2\imag\left(z_{Q}\right)\right)^{k}X_{z_{Q}}^{-k}(z) + O(1)\right).
	\]
	Hence, by Theorem~\ref{theorem fourier expansion} the Shintani theta lift of $f_{k,\mathcal{A}}$ is given by
	\begin{equation}\label{Pf}
	\Phi_{k}(f_{k,\mathcal{A}},\tau) = \frac{1}{2}\sum_{D > 0}\tr_{f_{k,\mathcal{A}}}(D)e^{2\pi i D\tau} - \frac{2^{2-k}\sqrt{|d|}}{(k-1)!\,\left|\overline{\Gamma}_{z_{\mathcal{A}}}\right|}\left[R_{2-2k,z}^{k-1}\left(\theta_{1-k}^{*}(z,\tau)\right)\right]_{z = z_{\mathcal{A}}}.
	\end{equation}
	Furthermore, a short calculation yields
	\begin{align}\label{eq lowering fkA}
	L(\Phi_{k}(f_{k,\mathcal{A}},\tau)) = \frac{2^{1-k}\sqrt{|d|}}{(k-1)!\,\left|\overline{\Gamma}_{z_{A}}\right|}\left[R_{2-2k,z}^{k-1}\left(\Theta_{k}^{*}(z,\tau)\right)\right]_{z = z_{\mathcal{A}}}.
	\end{align}

Following \cite{bringmannkanekohnen}, we define the function
\begin{align*}
 \mathcal{F}_{1-k,D}(z) := \frac{D^{\frac{1}{2}-k}}{\binom{2k-2}{k-1}\pi}\sum_{Q \in \mathcal{Q}_{D}}\sgn(Q_{z})Q(z,1)^{k-1}\psi\left( \frac{Dy^{2}}{|Q(z,1)|^{2}}\right),
\end{align*}
where $\psi(v) \! \!:=\! \!\frac{1}{2}\beta(v;k-\frac{1}{2},\frac{1}{2} )$ is a special value of the incomplete $\beta$-function. The function $\mathcal{F}_{1-k,D}$ is a locally harmonic Maass form (in the sense of \cite{bringmannkanekohnen})  of weight $2-2k$ for $\Gamma$ with jump singularities along the geodesics $C_{Q}$ for $Q \in\mathcal{Q}_{D}$. By Theorem 1.2 of \cite{bringmannkanekohnen} the function $\mathcal{F}_{1-k,D}$ maps to a constant multiple of $f_{k,D}$ under $\xi_{2-2k}$.

\begin{proposition}\label{proposition trace formula F1-kD}
	Under the assumptions of Theorem \ref{corollary trace formulas}, we have 
	\begin{align*}
	\sum_{D > 0}a_{F}(-D)\tr_{f_{k,\mathcal{A}}}(D) &= \frac{2^{k}\sqrt{|d|}}{(k-1)!\,\left|\overline{\Gamma}_{z_{\mathcal{A}}}\right|}\sum_{D > 0}a_{F}(-D)D^{k-\frac{1}{2}}R_{2-2k}^{k-1}\left(\mathcal{F}_{1-k,D}\right)(z_{\mathcal{A}}).
	\end{align*}
\end{proposition}

\begin{proof}
	Let $P_{\frac{3}{2}-k,D}$ be the unique harmonic Maass form of weight $\frac{3}{2}-k$ for $\Gamma_{0}(4)$ satisfying the Kohnen plus space condition which maps to a cusp form of weight $k+\frac{1}{2}$ under the $\xi$-operator and which has a Fourier expansion of the form $\frac{2}{3}e^{-2\pi i D\tau } + O(1)$ at $i \infty$ (see \cite{bringmannono} for its explicit construction as a Poincar\'e series). Then we obtain that $F = \frac{3}{2}\sum_{D > 0}a_{F}(-D)P_{\frac{3}{2}-k,D}$. By Theorem~1.3~(2) of \cite{brikavia}, we have
	\begin{align}\label{eq integral PD 1}
	2^{2k-2}D^{k-\frac{1}{2}}\mathcal{F}_{1-k,D}(z) = \lim_{T \to \infty}\int_{\mathcal{F}_{T}(4)}P_{\frac{3}{2}-k,D}(\tau)\Theta_{k}^{*}(z,\tau)\frac{du dv}{v^{2}},
	\end{align}
	where 
	\[
	\mathcal{F}_{T}(4) := \bigcup_{M \in \Gamma_{0}(4)\backslash \Gamma}M\mathcal{F}_{T}, \qquad \mathcal{F}_{T} := \left\{\tau = u+iv \in \H: |u| \leq \frac{1}{2}, |\tau| \geq 1, v \leq T\right\},
	\]
	is a truncated fundamental domain for $\Gamma_{0}(4)\backslash \H$. Note that the normalization of the functions $P_{\frac{3}{2}-k,D}$ and $\mathcal{F}_{1-k,D}$ in \cite{brikavia} differs from the normalization used in this paper, which explains the different constants when comparing equation \eqref{eq integral PD 1} to Theorem~1.3~(2) in \cite{brikavia}.

 Now we apply the iterated raising operator and plug $z = z_{\mathcal{A}}$ into \eqref{eq integral PD 1}. A standard argument involving the dominated convergence theorem shows that, for $z_{\mathcal{A}}$ not lying on any of the geodesics $C_{Q}$ for $Q \in \mathcal{Q}_{D}$, we have
	\begin{align*}
	2^{2k-2}D^{k-\frac{1}{2}}R_{2-2k}^{k-1}\left(\mathcal{F}_{1-k,D}\right)(z_{\mathcal{A}}) = &\lim_{T \to \infty}\int_{\mathcal{F}_{T}(4)}P_{\frac{3}{2}-k,D}(\tau)\left[R_{2-2k,z}^{k-1}\left(\Theta_{k}^{*}(z,\tau)\right)\right]_{z = z_{\mathcal{A}}}\frac{du dv}{v^{2}}.
	\end{align*}
	Using \eqref{eq lowering fkA}, Stokes' Theorem in the form given in Lemma~2 in \cite{ditreal}, and the fact that $F = \frac{3}{2}\sum_{D > 0}a_{F}(-D)P_{\frac{3}{2}-k,D}$ is holomorphic on $\H$, we obtain
	\begin{align} \label{inteo}
	&\frac{2^{k-1}\sqrt{|d|}}{(k-1)!\,\left|\overline{\Gamma}_{z_{\mathcal{A}}}\right|}\sum_{D > 0}a_{F}(-D)D^{k-\frac{1}{2}}R_{2-2k}^{k-1}\left(\mathcal{F}_{1-k,D}\right)(z_{\mathcal{A}})\notag \\
	&\quad =\frac{2}{3} \lim_{T \to \infty}\int_{\mathcal{F}_{T}(4)}F(\tau)\overline{\xi_{k+\frac{1}{2}}\left(\Phi_{k}(f_{k,\mathcal{A}},
	\tau)\right)}v^{\frac{3}{2}-k}\frac{dudv}{v^{2}}\notag \\
	&\quad = \frac{2}{3} \lim_{T \to \infty}\int_{-\frac{1}{2}+iT}^{\frac{1}{2}+iT}\left(F(\tau)\Phi_{k}(f_{k,\mathcal{A}},\tau)+ \frac{1}{2}F^{\even}(\tau)\Phi_{k}^{\even}(f_{k,\mathcal{A}},\tau) + \frac{1}{2}F^{\odd}(\tau)\Phi_{k}^{\odd}(f_{k,\mathcal{A}},\tau)\right) d\tau.
	\end{align}
	Here we set, for $f(\tau) = \sum_{D \in \Z}a(v,D)e^{2\pi i D \tau}$,
	\[
	f^{\even}(\tau) := \sum_{\substack{D \in \Z \\ D \equiv  { 0 \pmod{2}  }}}a\left(\frac{v}{4},D\right)e^{2\pi iD\frac{\tau}{4}}, \qquad f^{\odd}(\tau) := \sum_{\substack{D \in \Z \\ D \equiv {1 \pmod{2} }}}a\left(\frac{v}{4},D\right)e^{2\pi i \frac{D}{8}}e^{2\pi i D\frac{\tau}{4}}.
	\]
	The integral in \eqref{inteo} picks out the constant term in the Fourier expansion of the integrand. Using \eqref{Pf} we get that \eqref{inteo} equals 
	\begin{multline}\label{eq integral fkA 2}
	 \frac{1}{2}\sum_{D > 0}a_{F}(-D)\tr_{f_{k,\mathcal{A}}}(D) - \!\!\lim_{T \to \infty} \frac{2^{3-k}\sqrt{|d|}}{3(k-1)!\,\left|\overline{\Gamma}_{z_{\mathcal{A}}}\right|}\!\!\sum_{D \in \Z}a_{F}(-D)\!\!\sum_{Q \in \mathcal{Q}_{D}}\!\!\left[R_{2-2k,z}^{k-1}\left(\phi_{Q}(z,T) \right)\right]_{z = z_{\mathcal{A}}}\\
	 - \!\!\lim_{T \to \infty} \frac{2^{2-k}\sqrt{|d|}}{3(k-1)!\,\left|\overline{\Gamma}_{z_{\mathcal{A}}}\right|}\!\!\sum_{D \in \Z}a_{F}(-D)\!\!\sum_{Q \in \mathcal{Q}_{D}}\!\!\left[R_{2-2k,z}^{k-1}\left(\phi_{Q}\left(z,\tfrac{T}{4}\right)\right)\right]_{z = z_{\mathcal{A}}}\!\!.
	\end{multline}
	One can show that
	\[
	\left[R_{2-2k,z}^{k-1}\left(\phi_{Q}(z,T) \right)\right]_{z = z_{\mathcal{A}}} = P_{1}\left(\sqrt{T}\right)\erfc\left(2\sqrt{\pi T}|Q_{z_{\mathcal{A}}}|\right) + P_{2}\left(\sqrt{T}\right)e^{-4\pi T Q_{z_{\mathcal{A}}}^{2}},
	\]
	where $P_{1}$ and $P_{2}$ are polynomials. For $D \leq 0$ and $Q \in\mathcal{Q}_{D}\backslash \{[0,0,0]\}$ we have $Q_{z}^{2} > 0$ for all $z \in \H$, for $Q = [0,0,0]$ we have $\phi_{Q}(z,v) = 0$ by definition, and for those $Q \in \mathcal{Q}_{D}$ with $D > 0$ appearing above (i.e., $a_{F}(-D) \neq 0$) we have $Q_{z_{\mathcal{A}}}^{2} > 0$ since $z_{\mathcal{A}}$ does not lie on any geodesic $C_{Q}$ for $Q \in \mathcal{Q}_{D}$ with $a_{F}(-D) \neq 0$, by assumption. Since $\erfc(C\sqrt{T})$ and $e^{-CT}$ for $C > 0$ are rapidly decreasing as $T \to \infty$, the limits on the right-hand side of \eqref{eq integral fkA 2} vanish, and we obtain the formula stated in the proposition.
\end{proof}

	By Theorem~7.1 in \cite{bringmannkanekohnen}, the function $\mathcal{F}_{1-k,D}$ is locally a polynomial if $S_{2k} = \{0\}$. More generally, by taking suitable linear combinations, we get the following result.

	\begin{lemma}\label{lemma local polynomial} Assuming the conditions as in Theorem~\ref{corollary trace formulas}, we have
	\begin{multline}\label{eq local polynomial}
	\sum_{D > 0}a_{F}(-D)D^{k-\frac{1}{2}}\mathcal{F}_{1-k,D}(z)\\
	= \sum_{D > 0}a_{F}(-D)\left(-\frac{c_{k}(D)}{2^{k}\binom{2k-2}{k-1}}+ 2^{3-2k}\!\!\!\!\!\sum_{\substack{Q =[a,b,c]\in \mathcal{Q}_{D} \\ Q_{z} > 0 > a}}Q(z,1)^{k-1}\right),
	\end{multline}	
	where $c_{k}(D)$ is the constant defined in \eqref{eq ckD}.
	\end{lemma}
	
	\begin{proof}
		Recall that the non-holomorphic and holomorphic Eichler integrals of a cusp form $f(\tau) = \sum_{n\geq 1}a_{f}(n)q^{n} \in S_{2k}$ are defined by
		\[
		f^{*}(z) := (-2i)^{1-2k}\int_{-\overline{z}}^{i\infty}\overline{f(-\overline{\tau})}(\tau+z)^{2k-2}d\tau, \qquad \mathcal{E}_{f}(z) := \sum_{n\geq1}\frac{a_{f}(n)}{n^{2k-1}}q^{n}.
		\]	
		By Theorem 7.1 of \cite{bringmannkanekohnen} the locally harmonic Maass form $\mathcal{F}_{1-k,D}(z)$ decomposes as a sum
		\begin{align}\label{FDecomposition}
		\mathcal{F}_{1-k,D}(z) = \frac{D^{-\frac{k+1}{2}}}{\binom{2k-2}{k-1}}f_{k,D}^{*}(z)+D^{-\frac{k+1}{2}}\frac{(k-1)!^{2}}{(4\pi)^{2k-1}}\mathcal{E}_{f_{k,D}}(z) + \mathcal{P}_{1-k,D}(z),
		\end{align}
		where
		\[
		\mathcal{P}_{1-k,D}(z) := -\frac{c_{k}(D)}{2^{k}\binom{2k-2}{k-1}}+ 2^{3-2k}\!\!\!\!\!\sum_{\substack{Q =[a,b,c]\in \mathcal{Q}_{D} \\ Q_{z} > 0 > a}}Q(z,1)^{k-1}
		\]
		is locally a polynomial. Note that the constant $c_{\infty}$ defined in (7.3) of \cite{bringmannkanekohnen} is missing a factor $2$, which is corrected in the above formula.

		By Theorem~1.1~(2) of \cite{brikavia}, we have
		\[
		\frac{2^{2k-3}}{6\binom{2k-2}{k-1}}D^{\frac{k}{2}-1}f_{k,D}(z) = \Phi_{k}^{*}\left(\xi_{\frac{3}{2}-k}\left(P_{\frac{3}{2}-k,D}\right),z\right),
		\]
		where
		\[
		\Phi_{k}^{*}(f,z) : =\frac{1}{6}\lim_{T \to \infty}\int_{\mathcal{F}_{T}(4)}f(\tau)\overline{\Theta_{k}(z,\tau)}v^{k+\frac{1}{2}}\frac{dudv}{v^{2}}
		\]
		is the Shimura theta lift of a cusp form $f$ of weight $k+\frac{1}{2}$ for $\Gamma_{0}(4)$. Using that $F = \frac{3}{2}\sum_{D > 0}a_{F}(-D)P_{\frac{3}{2}-k,D}$ is holomorphic on $\H$, we find that
		\begin{align*}
		\frac{2^{2k-3}}{6\binom{2k-2}{k-1}}\sum_{D > 0}a_{F}(-D)D^{\frac{k}{2}-1}f_{k,D}(z) &= \Phi_{k}^{*}\left(\xi_{\frac{3}{2}-k}\left(\sum_{D > 0}a_{F}(-D)P_{\frac{3}{2}-k,D}\right),z\right) = 0,
		\end{align*}
		which implies that the analogous linear combinations of the non-holomorphic and holomorphic Eichler integrals $f_{k,D}^{*}$ and $\mathcal{E}_{f_{k,D}}$ in the decomposition \eqref{FDecomposition} vanish. This yields the stated formula.
		\end{proof}
	
	In view of Proposition~\ref{proposition trace formula F1-kD}, we need to apply the iterated raising operator $R_{2-2k}^{k-1}$ to the right-hand side of \eqref{eq local polynomial}. The following lemma is well-known and not hard to prove.
	\begin{lemma}\label{lemma raising constant}
		For $k \in \N$ we have
		\[
		R_{2-2k}^{k-1}(1) = (-1)^{k+1}(k-1)!\binom{2k-2}{k-1} y^{1-k}.
		\]
	\end{lemma}
	
	Next, we compute the iterated raising operator applied to $Q(z,1)^{k-1}$.

	\begin{lemma}\label{lemma raising Q}
		For $k \in \N$ and $Q \in \mathcal{Q}_{D}$ with $D > 0$ we have
		\begin{align*}
		R_{2-2k}^{k-1}\left( Q(z,1)^{k-1}\right) = \left(2i\sqrt{D}\right)^{k-1}(k-1)!P_{k-1}\left(\frac{iQ_{z}}{\sqrt{D}} \right).
		\end{align*}
	\end{lemma}

	\begin{proof}
		Since $D > 0$ there exists a matrix $M \in \SL_{2}(\R)$ such that $Q\circ M= [0,\sqrt{D},0]$. Using that $Q(z,1)|_{-2}M = (Q\circ M)(z,1), Q_{Mz} = (Q\circ M)_{z}$, and the fact that the slash operator commutes with the raising operator, we can assume without loss of generality  that $Q = [0,\sqrt{D},0]$. Then $Q(z,1) = \sqrt{D}z$ and $Q_{z} = \sqrt{D}\frac{x}{y}$. We rewrite the iterated raising operator using formula (56) in Zagier's part in \cite{zagier123} to get
		\begin{align*}
		R_{2-2k}^{k-1}\left(z^{k-1}\right) &= (-4\pi)^{k-1}\sum_{j = 0}^{k-1}(-1)^{k+1+j}\binom{k-1}{j}\frac{(2-2k+j)_{k-1-j}}{(4\pi y)^{k-1-j}(2\pi i)^{j}}\frac{\partial^{j}}{\partial z^{j}}z^{k-1} \\
		&=\sum_{j = 0}^{k-1} (2i)^j\binom{k-1}{j}\frac{(2-2k+j)_{k-1-j}}{y^{k-1-j}}\frac{(k-1)!}{(k-1-j)!}z^{k-1-j},
		\end{align*}
		where $(a)_{m} := a(a+1)\cdots (a+m-1)$ is the Pochhammer symbol. We replace $j \mapsto k-1-j$ and obtain, after some simplification, that this equals 
		\begin{align*}
		&\left(2i\right)^{k-1}(k-1)!\sum_{j = 0}^{k-1}\binom{k-1}{j}\binom{k-1+j}{j}\left(\frac{\frac{ix}{y}-1}{2}\right)^{j} = \left(2i\right)^{k-1}(k-1)!P_{k-1}\left(\frac{ix}{y} \right),
		\end{align*} 
		where we use a formula for the Legendre polynomials which can be obtained by combining (22.3.2) and (22.5.24) in \cite{abramowitz}. Recalling that $\frac{ix}{y} = \frac{iQ_{z}}{\sqrt{D}}$ finishes the proof.
	\end{proof}
	We are now ready to prove Theorem~\ref{corollary trace formulas} and Theorem~\ref{theorem trace formulas}.

	\begin{proof}[Proof of Theorem~\ref{theorem trace formulas}] By combining Proposition~\ref{proposition trace formula F1-kD}, Lemma~\ref{lemma local polynomial}, Lemma~\ref{lemma raising constant}, and Lemma~\ref{lemma raising Q}, we obtain the formulas given in Theorem~\ref{theorem trace formulas}.
	\end{proof}
	
	\begin{proof}[Proof of Theorem~{\ref{corollary trace formulas}}] From the functional equation of the Dirichlet $L$-function and its evalutation at negative integers in terms of Bernoulli polynomials $B_{k}(x) \in \Q[x]$ (see \S7 in \cite{zagierbook}), and the well-known evaluation of the Riemann zeta function at even natural numbers in terms of Bernoulli numbers $B_{k} \in \Q$, we obtain the rational number
	\begin{align}\label{eq rational formula ckD}
	c_{k}(D) = -\frac{f^{2k-1}D_{0}^{k-1}\binom{2k}{k}B_{k}}{2^{k-2}(2k-1)B_{2k}}\sum_{\ell = 1}^{D_{0}}\left(\frac{D_{0}}{\ell} \right)B_{k}\left( \frac{\ell}{D_{0}}\right)\sum_{m \mid f}\mu(m)\left(\frac{D_{0}}{m} \right)m^{-k}\sigma_{1-2k}\left(\frac{f}{m}\right).
	\end{align}
	Furthermore, we have that $|z_{\mathcal{A}}|^{2}, x_{\mathcal{A}} \in \Q$ and $y_{\mathcal{A}} \in \sqrt{|d|}\Q$, hence $Q_{z_{\mathcal{A}}} \in \sqrt{|d|}\Q$. This implies that $\sum_{D > 0}a_{F}(-D)\tr_{f_{k,\mathcal{A}}}(D) \in \Q$, finishing the proof of Theorem~\ref{corollary trace formulas}.
	\end{proof}
	
	\section{Numerical evaluation of $\tr_{f_{k,\mathcal{A}}}(D)$}\label{section numerical evaluation}
	
	In order to emphasize the explicit nature of the rational formulas given in Theorem~\ref{theorem trace formulas}, we give some details on their numerical evaluation. The constant $c_{k}(D)$ can be computed using \eqref{eq rational formula ckD}. To evaluate the sum over $Q = [a,b,c] \in \mathcal{Q}_{D}$ appearing in the formulas, recall that $Q_{z_{A}} > 0 > a$ implies that the CM point $z_{A} = x_{A}+iy_{A}$ lies in the interior of the bounded component of $\H \setminus C_{Q}$. This can only happen if $x_{A}$ lies between the two real endpoints of the semi-circle $C_{Q}$ and if $y_{A}$ is smaller than the radius of $C_{Q}$. It is not hard to see that for $a < 0$ this implies the conditions
	\begin{align}\label{eq inequalities}
	2|a|x_{A}-\sqrt{D} \leq b \leq 2|a|x_{A}+\sqrt{D},\qquad |a| \leq \frac{\sqrt{D}}{2y_{A}}, \qquad c:=\frac{D-b^{2}}{4|a|} \in \Z.
	\end{align}
	There are only finitely many integral binary quadratic forms $[a,b,c]$ with $a < 0$ satisfying these conditions. Hence it remains to check whether $Q_{z_{A}} > 0$ for these finitely many quadratic forms.
	
	For example, let $k = 2, \mathcal{A}= [A]$ with $A = [1,1,1]$, and $D = 5$. Then we have $d = -3, z_{A} = \frac{-1+\sqrt{3}}{2} = e^{\frac{2\pi i}{3}},$ and $|\overline{\Gamma}_{z_{\mathcal{A}}}| = 3$. A short calculation gives that $c_{2}(5) = 8$. The conditions \eqref{eq inequalities} are only satisfied by the four quadratic forms $[-1, \pm 1, 1], [-1, \pm 3, -1]$, and only $Q = [-1,-1,1]$ satisfies $Q_{z_{A}} > 0$. Hence the sum in $\tr_{f_{2,[A]}}(5)$ has only one summand, whose value is $[-1,-1,1]_{z_{A}} = \frac{1}{\sqrt{3}}$. Altogether, we obtain $\tr_{f_{2,[A]}}(5) =4$.


\begin{thebibliography}{99}
	\bibitem{abramowitz} M.~Abramowitz and I.~A.~Stegun, {\it Handbook of mathematical functions with formulas, graphs, and mathematical tables}, National Bureau of Standards Applied Mathematics Series {\bf 55} (1964).
	\bibitem{ans} C.~Alfes-Neumann and M. Schwagenscheidt, {\it Shintani theta lifts of harmonic {M}aass forms}, preprint \texttt{arXiv:1712.04491}.
	\bibitem{borcherds} R.~Borcherds, {\it Automorphic forms with singularities on {G}rassmannians}, Invent. Math. {\bf 132} (1998), 491--562.
	\bibitem{bringmannguerzhoykane} K.~Bringmann, P.~Guerzhoy, and B.~Kane, {\it On cycle integrals of weakly holomorphic modular forms}, Math. Proc. Cambridge Philos. Soc. {\bf 158} (2015), 439--449.
	\bibitem{bringmannkanekohnen} K.~Bringmann, B.~Kane, and W.~Kohnen, {\it Locally harmonic {M}aass forms and the kernel of the {S}hintani lift}, Int. Math. Res. Notices {\bf 1} (2015), 3185--3224.
	\bibitem{bringmannkanevonpippich} K.~Bringmann, B.~Kane, and A.~ von Pippich, {\it Regularized inner products of meromorphic modular forms and higher {G}reen's Functions}, to appear in Commun. Contemp. Math. (2018).
	\bibitem{brikavia} K.~Bringmann, B.~Kane, and M.~Viazovska, {\it Theta lifts and local {M}aass forms}, Math. Res. Lett. {\bf 20} (2013), 213--234.
	\bibitem{bringmannkanezwegers} K.~Bringmann, B.~Kane, and S.~Zwegers, {\it On a completed generating function of locally harmonic Maass forms}, Compositio Math. \textbf{150} (2014), 749--762.
	\bibitem{bringmannono} K.~Bringmann and K.~Ono, {\it {A}rithmetic properties of coefficients of half--integral weight {M}aass--{P}oincar\'e series}, Math. Ann. {\bf 337} (2007), 591--612.
	\bibitem{bruinierhabil} J.~Bruinier, {\it Borcherds products on {O}(2, {$l$}) and {C}hern classes of
		{H}eegner divisors}, Lecture Notes in Mathematics {\bf 1780}, Springer-Verlag, Berlin (2002).
	\bibitem{bifl} J.~Bruinier, J.~Funke, \"O.~Imamo\={g}lu, and Y.~Li, {\it Modularity of generating series of winding numbers}, Research in the Math. Sci. {\bf 5} (2018), 23.
	\bibitem{zagier123} J.~Bruinier, G.~van der Geer, G.~Harder, and D.~Zagier, {\it The 1-2-3 of modular forms}, Lectures from the Summer School on Modular Forms and their Applications held in Nordfjordeid, June 2004, Edited by Kristian Ranestad, Springer-Verlag, Berlin (2008).
	\bibitem{ditreal} W.~Duke, \"O.~Imamo\={g}lu, and \'A.~T\'oth, {\it Real quadratic analogues of traces of singular invariants}, Int. Math. Res. Notices {\bf 13} (2011), 3082--3094.
	\bibitem{harveymoore} J.~Harvey, and G.~Moore, {\it Algebras, {BPS} states, and strings}, Nuclear Phys. B {\bf 463} (1996), 315--368.
%	\bibitem{jkk} D.~Jeon, S.-Y. Kang, and C. H. Kim {\it Weak {M}aass--{P}oincar\'e series and weight $3/2$ mock modular forms}, J. Number Theory {\bf 133} (2013), 2567--2587.
	\bibitem{kohnenfouriercoefficients} W.~Kohnen {\it Fourier coefficients of modular forms of half-integral weight}, Math. Ann. {\bf 271} (1985), 237--268.
		\bibitem{kohnenzagierrationalperiods} W.~Kohnen and D.~Zagier, {\it Modular forms with rational periods} in ``Modular forms'', ed. by R.~A.~Rankin, Ellis Horwood, (1985), 197--249.
	\bibitem{niwa} S.~Niwa. {\it Modular forms of half integral weight and the integral of certain theta-functions}, Nagoya Math. J. {\bf 56} (1974), 147--161.
	\bibitem{petersson50} H.~Petersson, {\it {K}onstruktion der {M}odulformen und der zu gewissen {G}renzkreisgruppen geh\"origen automorphen {F}ormen von positiver reeller {D}imension und die vollst\"andige {B}estimmung ihrer {F}ourierkoeffzienten}, S.-B. Heidelberger Akad. Wiss. Math. Nat. Kl. (1950), 415--474.
	\bibitem{petersson} H.~Petersson, {\it {\"U}ber automorphe {O}rthogonalfunktionen und die {K}onstruktion der automorphen {F}ormen von positiver reeller {D}imension}, Math. Ann. {\bf 127} (1954), 33--81.
	\bibitem{shimura} G.~Shimura, {\it On modular forms of half integral weight}, Ann. of Math. (2) {\bf 97} (1973), 440--481.
	\bibitem{shintani} T.~Shintani, {\it On construction of holomorphic cusp forms of half integral weight}, Nagoya Math. J. {\bf 58} (1975), 83--126.
	\bibitem{zagierbook} D.~Zagier, {\it {Z}etafunktionen und quadratische {K}{\"o}rper: Eine {E}inf{\"u}hrung in die h{\"o}here {Z}ahlentheorie}. Hochschultext (Berlin), Springer-Verlag (1981).
	\bibitem{zagierrealquadratic} D.~Zagier, {\it Modular forms associated to real quadratic fields}. Invent. math. (1) {\bf 30} (1975), 1--46.
	\bibitem{zemel} S.~Zemel, {\it Regularized pairings of meromorphic modular forms and theta lifts}, J. Number Theory {\bf 162} (2016), 275--311. 
\end{thebibliography}
\end{document}